\documentclass{amsart}
\usepackage{amsmath,amsfonts,amssymb,amsthm}
\usepackage{enumerate}
\usepackage[usenames]{color}
\usepackage[english]{babel}

 \newtheorem{theorem-main}{\bf Theorem}
\newtheorem{theorem}{\bf Theorem}
 \newtheorem{lemma}[theorem]{\bf Lemma}
 \newtheorem{proposition}[theorem]{\bf Proposition}
 \newtheorem{definition}[theorem]{\bf Definition}

    \newtheorem{rks}[theorem]{\bf Remarks}

 \newcommand{\R}{\mathbb{R}}
 \newcommand{\Q}{\mathbb{Q}}
 \newcommand{\N}{\mathbb{N}}
 \newcommand{\Z}{\mathbb{Z}}
  \newcommand{\C}{\mathbb{C}}

 \newcommand{\MM}{\mathcal{M}}

\newcommand{\vp}{\varphi}

\newcommand{\wt}[1]{{{\widetilde{#1}}}}

\newcommand{\g}{\gamma}

\renewcommand{\wt}[1]{\widetilde{#1}}

 \renewenvironment{proof}{{\em \noindent Proof.}}
 {\hfill $\square$\newline}

 \newcommand{\obra}[3]{{\sc #1} {\em #2}. {#3}.}

\title[Real Turrittin Theorem]
{Turrittin's Theorem revisited. The real case}
\author{M. Barkatou}
\address{XLIM, Universit\'e de Limoges ;  CNRS,\\123, Av.  A. Thomas, 87060 Limoges cedex, France}
\email{moulay.barkatou@unilim.fr}
\author{F. A. Carnicero}
\address{Departamento de \'{A}lgebra, An\'{a}lisis Matem\'{a}tico, Geometr\'{\i}a y Topolog\'{\i}a. 
	Universidad de Valladolid.
	Facultad de Ciencias. Paseo de Belén, 7, E-47011. Valladolid, Spain}
\email{carnicero\_fac@hotmail.com}
\author{F. Sanz Sánchez}
\address{Departamento de \'{A}lgebra, An\'{a}lisis Matem\'{a}tico, Geometr\'{\i}a y Topolog\'{\i}a. 
	Universidad de Valladolid.
	Facultad de Ciencias. Paseo de Belén, 7, E-47011. Valladolid, Spain} \email{fsanz@uva.es}
\thanks{This work was supported by Ministerio de Ciencia e Innovaci\'on (MTM2016-77642-C2-1-P and PID2019-105621GB-I00) and by Junta de Castilla y León (VA083G19). 
The first author thanks UVa for the support during several research stays at the Departamento de \'Algebra, Análisis Matemático, Geometría y Topología.}
\date{}


\begin{document}
\maketitle

\begin{abstract}
We establish a real version of Turrittin's result on polynomial and formal normal forms of linear systems of ODEs with meromorphic coefficients. Both the normal forms or the transformations used have only real coefficients. In order to adapt the proofs to the real case, we make a review of the result in the complex case.
\end{abstract}

\section{Preliminaries and statements}

Let $K$ be a field of characteristic zero and let $L_K=K[[x]][x^{-1}]$ be the field of formal meromorphic series with coefficients in $K$, endowed with the usual derivation with respect to $x$ (denoted only by a prime), and the usual valuation $\nu:L_K\to\Z\cup\{\infty\}$ defined as the minimum of the support of the series, also called the {\em order}. 
As a matter of notation, if $R$ is any ring and $n\in\N_{\ge 1}$, $\MM_n(R)$ denotes the ring of square matrices of size $n$ with entries in $R$.

A matrix $A\in\mathcal{M}_n(L_K)$ is identified  with the {\em formal meromorphic linear system of ODEs} 
$$
[A]\;\;\;\;\;Y'=AY,
$$
where $Y=(Y_1,...,Y_n)^t$ is a column vector of $n$ variables. Define {\em the order of $A$} to be $\nu(A):=\min\{\nu(a_{ij})\,:\,1\le i,j\le n\}$, where $A=(a_{ij})$. Sometimes we use the notation $A=A(x)$ to make explicit that we are dealing with meromorphic series in the variable $x$. Correspondingly, we will usually write the system as a series of matrices in the form
\begin{equation}\label{eq:system-develop}
A=x^{\nu(A)}(A_0+xA_1+\cdots),\mbox{ where }  A_i\in\mathcal{M}_n(K)\mbox{ for any }i\mbox{ and }A_0\ne 0.
\end{equation}
Also, if $N$ is a non-negative integer, the {\em truncated system} up to degree $N$ is defined by
$$
J_NA:=x^{\nu(A)}(A_0+xA_1+\cdots+x^{N}A_{N}).
$$
The system $A$ is called {\em singular} (at $x=0$) if $\nu(A)<0$. The {\em Poincar\'{e} rank} of the system is defined as the non-negative integer $q=q(A):=\max\{-\nu(A)-1,0\}$. We usually rewrite $A$ as the system of formal linear ODEs
$$
x^{q+1}Y'=\wt{A}Y,\;\;\mbox{ where }\;\;\wt{A}=A_0+xA_1+\cdots.
$$
A singular system with Poincaré rank $q=0$ (resp. $q>0$) is usually referred to be of {\em first kind} (resp. of {\em second kind}).

Denote by $I_n$ the identity matrix of size $n$. Define the {\em radiality index} of $A$ as the non-negative integer
$$
k=k(A):=\min(\{j\,:A_j\not\in K I_n\}\cup\{q\}).
$$
The truncation $J_{k-1}A$ is called the {\em radial part} of $A$. 

\strut

We are interested in the problem of getting formal normal forms of a given singular system under transformations of one of the following types:
\begin{enumerate}[(i)]
	\item Given $P\in GL_n(L_K)$, 
	the linear change of variables
	$
	Y=PZ$ transforms the system $[A]: Y'=AY$ into the system $[B]: Z'=BZ$ where 
	$$
	B=P^{-1}AP-P^{-1}P'.
	$$
	The map $
	\Psi_P:\MM_n(L_K)\to\MM_n(L_K)
	$ sending $A$ to $\Psi_P[A]:=P^{-1}AP-P^{-1}P'$ is bijective. It is called
	the {\em gauge transformation associated to $P$}. 
	A gauge transformation $\Psi_P$ will be called: {\em regular}, if $\nu(P)=0$ and $\det P(0)\ne 0$; {\em polynomial}, if  each entry of $P$ belongs to $K[x]$; {\em diagonal monomial}, if $P=\mbox{diag}(x^{k_1},\cdots,x^{k_n})$ with $k_j\in\N_{\ge 0}$ for each $j$. 
    \item Given $r\in\N_{\ge 1}$,  the change of the independent variable
    $
    x=z^r,
    $
    transforms the system $\frac{dY}{dx}=A(x)Y$ into a system $\frac{dY}{dz}=B(z)Y$, where 
    $$
    B(z)=rz^{r-1}A(z^r).
    $$
    Re-written with the same letter $x$, we define the map $R_r:\MM_n(L_K)\to\MM_n(L_K)$ given by $R_r[A]:=rx^{r-1}A(x^r)$, called the {\em ramification of order $r$}. It is an injective map but not bijective for $r>1$. 
\end{enumerate}

In addition, we are interested in polynomial (truncated) normal forms obtained by means of polynomial gauge transformations and ramifications (so that, if the initial system is polynomial or convergent, we preserve this character).  

\strut

The case $K=\C$ (or more generally, $K$ algebraically closed) is classical and treated with different approaches in the literature (Birkhoff \cite{Bir}, Hukuhara \cite{Huk1,Huk2}, Turrittin \cite{Tur2}, Wasow \cite{Was}, Moser \cite{Mos}, Balser-Jurkart-Lutz \cite{Bal-J-L}, Babbitt-Varadarajan \cite{Bab-V}, Barkatou \cite{Bar1,Bar2}, Barkatou-Pfl\"{u}gel \cite{Bar-P}. The different avatars of the algorithms for obtaining normal forms are commonly referred (as we will do here) by the generic expression {\em Turrittin's Theorem}.

Our objective in this paper is to extend Turrittin's Theorem to the real case $K=\R$, or more generally to the case where $K$ is a real closed field. 
As far as we know, this case has not been treated yet (except, of course, in the situation of a constant system $A\in\MM_n(K)$ for which the usual well known real Jordan canonical form of $A$  was proposed by Turrittin himself in \cite{Tur3}). We present versions of real (formal and polynomial) normal forms for any system, in such a way that they can be obtained by transformations written in the base field $K$, without passing through the algebraic closure $\overline{K}=K(\sqrt{-1})$. 

\subsection{The complex case}
In order to make precise statements and expose some of the steps that are useful to treat the real case, we propose first a brief revision of the complex case. Despite of its prevalence in the literature, we are led ourselves to sketch the different steps of the corresponding proofs (in section~\ref{sec:proofs-complex}), instead of simply addressing the reader to the references. There are additional reasons to justify this revision:

- Although there are other proofs (even better ones from the point of view of computational effectiveness, see \cite{Bar2}), maybe the most commonly used reference for the complex case is Wasow's book \cite{Was}. We decided to follow also this last reference here.  However, in that proof, the final arguments concerning the induction on the Poincar\'{e} rank is perhaps not sufficiently clarified: it drops as long as we do not need to make a ramification, but it increases after ramifications, an operation which is unavoidable in general. The required modification, even its simplicity, is worth to be made, in any case.

- In existing proofs of Turrittin's Theorem, it is frequently allowed the use of {\em exponential shiftings} when the leading matrix has a single eigenvalue. Such transformations have not an algebraic or formal nature and are ``strange'' to the initial setting of the systems. Although they commute with the whole matrix of the system so that the resulting system has also formal meromorphic coefficients, the exponential shiftings may behave very badly with respect to non-linear terms in general systems. Thus, for applications, it is better to avoid these operations.

- The search of precise statements for polynomial normal forms of Turrittin's result make necessary to enter in some details of the proofs; such statements are not exactly pursued in the common references, mostly devoted to obtain expressions of a fundamental matrix of solutions (cf. Remark~\ref{rks}, (b) below).

\strut

We start by defining the normal forms that we expect to get. 

\begin{definition}[Turrittin-Ramis-Sibuya form]\label{def:TRS-form}
	Let $A\in\MM_n(L_K)$ be a system with Poincaré rank $q=q(A)$ and let $\mu\in\N_{\ge 0}$. We say that $A$ is in {\em Turrittin-Ramis-Sibuya form of degree $\mu$ (and of rank $q$)}, or {\em in $(TRS)_\mu^q$-form} for short, if it is written as
	$$
	A(x)=x^{-(q+1)}\left(D(x)+x^qC+O(x^{q+\mu+1})\right),
	$$ 
	where $D(x)={\rm diag}(d_1(x),...,d_n(x))$ is a diagonal matrix with polynomial entries $d_j(x)\in K[x]$ of degree at most $q-1$ (equal zero iff $q=0$) and $C\in\mathcal{M}_n(K)$ is a constant matrix commuting with $D(x)$. In this case, the truncated system $J_q(A)=x^{-(q+1)}(D(x)+x^qC)$ is called the {\em principal part} of $A$, while $D(x)$, resp. $C$, is called the {\em exponential part}, resp. the {\em residual matrix}.
\end{definition}
Notice that if the system $A$ is singular of first kind (that is $q=0$) then $A$ is already in $(TRS)_0^0$-form, with exponential part equal to $D(x)=0$ and residual matrix $C=A_0$. On the other hand, if $A$ is in $(TRS)_\mu^q$-form for some $\mu$ and $q>0$  then its exponential part $D(x)$ is not zero; in fact, it satisfies $D(0)\ne 0$.

The names Ramis and Sibuya in the definition above come from those authors paper \cite{Ram-S}, devoted to summability properties of formal solutions of systems of holomorphic ODEs where the linear part is in $(TRS)$-form of some degree. We have added the name Turrittin by obvious reasons. It is worth to notice that analogous expressions as the $(TRS)$-forms appear also in the context of germs of biholomorphisms in \cite{Lop-S,LRSV} (where the name of ``Ramis-Sibuya form'' is used). 

\begin{definition}\label{def:non-resonant matrix}
Let $C\in\MM_n(K)$ be a constant square matrix with entries in $K$. We say that $C$ is {\em non-resonant}\footnote{Other authors, for instance Balser in his book \cite{Bal}, use instead the terminology ``$C$ has good spectrum''.}  if for any pair of distinct eigenvalues $\lambda,\lambda'\in\overline{K}$ of $C$ we have $\lambda-\lambda'\not\in\Z$. 
\end{definition}

Now, Turrittin's results for the case where $K=\overline{K}$ can be stated in the two following theorems. 
\begin{theorem}[Complex Polynomial Normal Form]\label{th:c-turrittin-poly} 
		Suppose that $K$ is an algebraically closed field and let $A\in\MM_n(L_K)$ be a singular system with Poincaré rank equal to $q=q(A)$.
		\begin{enumerate}[(i)] 
		\item 
		There exist some $r\in\N_{\ge1}$ and finitely many polynomial gauge transformations $\vp_1,...,\vp_m$, either regular or diagonal monomial, such that, denoting $$\psi=\vp_m\circ\cdots\circ\vp_1\circ R_r,$$ the transformed system $\wt{A}=\psi[A]$ is in $(TRS)_0^{\tilde{q}}$-form, where $\tilde{q}=q(\wt{A})$. Moreover, if $B\in\MM_n(L_K)$ is another singular system with $q(B)=q$ and $J_{nq}A=J_{nq}B$ then $\wt{B}=\psi[B]$ is also in $(TRS)_0^{\tilde{q}}$-form with the same principal part as $\wt{A}$; i.e, $q(\wt{B})=\tilde{q}$ and
		$J_{\tilde{q}}\psi[A]=J_{\tilde{q}}\psi[B].$
		\item Assume that $A$ is in $(TRS)^q_0$-form and that its residual matrix is non-resonant. Then, for any given $\mu\ge 0$, there exists a regular polynomial gauge transformation $\phi_\mu=\Psi_{P^\mu}$, where $P^\mu(0)=I_n$, such that $\phi_\mu[\wt{A}]$ is in $(TRS)_\mu^{\tilde{q}}$-form with the same principal part as the original system $A$. Moreover, the family $\{P^\mu\}_\mu$ can be chosen such that $P^\mu$ is of degree at most $q+\mu$ and satisfying that $J_{q+\mu} P^{\mu'}=J_{q+\mu} P^{\mu}$ for any $\mu'>\mu$. 
		\item Assume that $A$ is in $(TRS)^q_0$-form. Then there exists a  gauge transformation $\phi$, given by a finite composition of regular polynomial or diagonal monomial transformations, such that $\phi[A]$ is in $(TRS)^q_0$-form with non-resonant residual matrix (and the same exponential part as $A$).  
	\end{enumerate}	
\end{theorem}

As a consequence of the theorem above, one obtains the following version of Turrittin's formal normal forms of complex meromorphic linear ODEs

\begin{theorem}[Complex Formal Normal Form]\label{th:c-turrittin-formal}
	Suppose that $K$ is an algebraically closed field and let $A\in\MM_n(L_K)$ be a singular system with Poincaré rank equal to $q=q(A)$.
		There exists a formal gauge transformation $\Psi_P$ and a ramification $R_r$ such that the transformed system $F=(\Psi_P\circ R_r)[A]$ has Poinaré rank equal to $\tilde{q}=q(F)$ and can be written in a {\em Formal Normal Form}
		$$
		\hspace{-1cm}[F]\hspace{1cm} Y'=x^{-(\tilde{q}+1)}(D(x)+x^{\tilde{q}}C)Y,
		$$
		where $D(x)$ and $C$ satisfy the requirements in Definition~\ref{def:TRS-form}; i.e., $F$ is in $(TRS)^{\tilde{q}}_0$-form and $J_{\tilde{q}}F=F$. Moreover, the transformation $\Psi_P$ can be chosen to be equal to $\Psi_P= \Psi_Q\circ\psi$, where $Q\in\MM_n(K[[x]])$ with $Q(0)=I_n$  and $\psi$ is a finite composition of regular polynomial or diagonal monomial gauge transformations.
 \end{theorem}

\begin{rks}\label{rks}
	{\em Concerning the statements in Theorem~\ref{th:c-turrittin-poly} and Theorem~\ref{th:c-turrittin-formal}, we have the following.
	\begin{enumerate}[(a)]
	\item The sufficient truncation order $nq$ in the second sentence of item (i) is already obtained for instance by Babbit-Varadarajan \cite{Bab-V} or Lutz-Sch\"afke \cite{Lut-S}.  Below, we propose a proof with the slightly improved order $N:=n(q-k)+k$, where $k$ is the radial index of the initial system $A$.
		\item  To obtain the formal normal form $[F]$ for the system $A$ in Theorem~\ref{th:c-turrittin-formal} is equivalent to say that there exists a matrix $P(t)\in\mathcal{M}_n(K[[t]])$ and some $r\in\N_{\ge 1}$ such that
\begin{equation}\label{eq:f-matrix-solutions}
	Z(x)=P(x^{1/r})\exp\left(
	\int\frac{D(x^{1/r})}
	{x^{(\tilde{q}+1))/r}}
	\right)x^\frac{C}{r}
\end{equation}
	is a {\em fundamental matrix of formal solutions} of the system $A$.
	\item Another consequence of the expression (\ref{eq:f-matrix-solutions}) is that the ratio $\tilde{q}/r$ and the exponential part $D(x)$, modulo ramification of $x$, are both invariant under formal meromorphic gauge transformations. More precisely, if we have two systems $A$ and $B$ such that $B=\Psi_T[A]$ with $T\in GL_n(K[[x]][x^{-1}])$ and we get (TRS)-normal forms of $A$ and $B$ as in item (i) of Theorem~\ref{th:c-turrittin-poly} with resulting Poincaré ranks $\tilde{q}_A$ and $\tilde{q}_B$ and exponential parts $D_A(x)$ and $D_B(x)$, respectively, then there are integers $r_1,r_2$ such that $\tilde{q}_A/r_1=\tilde{q}_B/r_2$ and $D_A(x^{1/r_1})=D_B(x^{1/r_2})$. In particular, $\tilde{q}=0$ iff the system $A$ is equivalent to a system of first kind under a formal gauge transformation. 
	\item In the proof proposed below, one could see that the sequence of gauge transformations used in items (i) or (iii) can be chosen so that any one of them, individually, do not increase the Poincaré rank of the system it applies to in the process. As we know, this observation only concerns the diagonal monomial gauge transformations, since a regular gauge transformations always preserves the Poincaré rank. 
	\item In the proof below, we propose a bound, in terms of the eigenvalues of the residual matrix $C$, for the degree of the transformation $\phi$ in item (iii). 
	\end{enumerate}
}
\end{rks}

\subsection{The real case}\label{sec:statement-real}

Suppose that $K$ is a real closed field, i.e., $K\varsubsetneq K(i)=\overline{K}$, where $i=\sqrt{-1}$. 
Given $\lambda=a+bi\in\overline{K}$, denote by 
$$
\Lambda_{\lambda}=\left(
\begin{array}{cc}
	a & -b \\
	b & a
\end{array}
\right). 
$$
Recall that the characteristic polynomial of $\Lambda_\lambda$ has roots $a\pm bi$ and is irreducible if and only if $\lambda\not\in K$, i.e., $b\ne 0$. 
For any $m\in\N_{\ge 1}$, define the monomorphism of $K$-algebras
$$
\Theta_m:\mathcal{M}_m(\overline{K})\to\mathcal{M}_{2m}(K),
$$
sending a matrix $C=(c_{uv})\in\mathcal{M}_m(\overline{K})$ to the $(2\times 2)$-block matrix $(\Lambda_{c_{uv}})\in\mathcal{M}_{2m}(K)$. A square matrix in the image of $\Theta_m$ will be called a {\em complex matrix over $K$}, or a {\em $\C$-matrix}, for short.

We extend $\Theta_m$ to a monomorphism of $K$-algebras, denoted with the same letter, from $\MM_m(L_{\overline{K}})$ into $\MM_{2m}(L_K)$; that is,
from formal meromorphic linear systems over $\overline{K}$ to formal meromorphic linear systems over $K$ of double dimension. A system in the image of this map will be called a {\em complex system (over $K$)} or a {\em $\C$-system}.

In what follows, if $U,V$ are two square matrices of sizes $k,l$, respectively, we denote by $U\oplus V$ the square matrix of size $k+l$ given in blocks
$$
U\oplus V:=\left(\begin{array}{cc} U&0\\0&V\end{array}\right).
$$

\begin{definition}[Real Turrittin-Ramis-Sibuya form]\label{def:TRS-form-real}
	Suppose thar $K$ is a real closed field.
	Let $A\in\MM_n(L_K)$ be a system with Poincaré rank $q=q(A)$ and let $\mu\in\N_{\ge 0}$. We say that $A$ is in {\em Real Turrittin-Ramis-Sibuya form of degree $\mu$ (and of rank $q$)}, or {\em in $(\R TRS)_\mu^q$-form} for short, if it can be written in the form
	$$
	A(x)=x^{-(q+1)}\left(D_1(x)\oplus D_2(x)+x^{q}(C_1\oplus C_2)+O(x^{q+\mu+1})\right),
	$$ 
	where 
		\begin{itemize}
		\item $D_1(x)={\rm diag}(e_1(x),...,e_{n_1}(x))$ is diagonal polynomial with entries $e_j(x)\in K[x]$ of degree at most $q-1$ (equal to zero if $q=0$).
		\item $D_2(x)=\Theta_{n_2}\left({\rm diag}(d_1(x),...,d_{n_2}(x))\right)$ is a diagonal $2\times 2$-block complex matrix such that the entries $d_j(x)$ belong to $\overline{K}[x]\setminus K[x]$ and are of degree at most $q-1$ (equal to zero if $q=0$). 
		\item $C_1$ and $C_2$ are constant matrices with entries in $K$ of sizes $n_1$ and $2n_2$, respectively, and $C_2$ is a $\C$-matrix.
		\item $[D_1(x),C_1]=0$ and $[D_2(x),C_2]=0$.
	\end{itemize}
	 In this case, the truncated system $J_q(A)=x^{-(q+1)}(D_1(x)\oplus D_2(x)+x^q(C_1\oplus C_2))$ is called the {\em principal part} of $A$, while $D_1(x)\oplus D_2(x)$, resp. $C_1\oplus C_2$, is called the {\em exponential part}, resp. the {\em residual matrix}.
\end{definition}

\strut

Our main result in the paper is the following real versions of Theorems~\ref{th:c-turrittin-poly} and \ref{th:c-turrittin-formal}.

\begin{theorem}[Real Polynomial Normal Forms]\label{th:r-turrittin-poly}
	Suppose that $K$ is a real closed field and let $A\in\MM_n(L_K)$ be a singular system with Poincaré rank equal to $q=q(A)$. 
	\begin{enumerate}[(i)]
		\item 	There exists  $r\in\N_{\ge1}$ and there exists finitely many polynomial gauge transformations $\vp_1,...,\vp_m$ (with coefficients in $K$), either regular or diagonal monomial, such that, denoting $$\psi=\vp_m\circ\cdots\circ\vp_1\circ R_r,$$ the transformed system $\wt{A}=\psi[A]$ is in $(\R TRS)_0^{\tilde{q}}$-form.	
		Moreover, if $B\in\MM_n(L_K)$ is another singular system with $q(B)=q$ and $J_{nq}A=J_{nq}B$ then, with the same transformation $\psi$, the transformed system $\wt{B}=\psi[B]$ is also in $(\R TRS)_0^{\tilde{q}}$ with the same principal part as $\wt{A}$, i.e.,
		$
		J_{\tilde{q}}\psi[A]=J_{\tilde{q}}\psi[B].
		$
		\item Assume that $A$ is in $(\R TRS)^q_0$-form and that its residual matrix is non-resonant. Then, for any $\mu\ge 0$ there exists a regular polynomial gauge transformation $\phi_\mu=\Psi_{P^\mu}$ where $P^\mu(0)=I_n$ such that $\phi_\mu[A]$ is in $(\R TRS)_\mu^{q}$-form and with the same principal part as the system $A$. Moreover, the family $\{P^\mu\}_\mu$ can be chosen such that $P^\mu$ is of degree at most $q+\mu$ and satisfying that $J_\mu P^{\mu'}=J_\mu P^{\mu}$ for any $\mu'>\mu$.
		\item Assume that $A$ is in $(\R TRS)^q_0$-form. Then there exists a polynomial gauge transformation $\phi$, given by a finite composition of regular or diagonal monomial transformations, such that $\phi[A]$ is in $(\R TRS)^q_0$-form  with non-resonant residual matrix (and the same exponential part as $A$). 
		\end{enumerate}
	\end{theorem}
\begin{theorem}\label{th:r-turrittin-formal}
Let $K$ be a real closed field and let $A\in\MM_n(L_K)$ be a singular system. Then there exists a formal gauge transformation $\psi_P$, with $P\in L_K$, and a ramification $R_r$ such that the transformed system $F^\R=(\psi_P\circ R_r)[A]$ has Poincaré rank equal to $\tilde{q}$ and is written as:
		$$
		\hspace{-1cm}[F^\R]\hspace{1cm} Y'=x^{-(\tilde{q}+1)}(D_1(x)\oplus D_2(x)+x^{\tilde{q}}(C_1\oplus C_2))Y,
		$$
		where $D_1,D_2,C_1,C_2$ satisfy the conditions in Definition~\ref{def:TRS-form-real}; that is $F^\R$ is in $(\R TRS)$-form and $F^\R=J_{\tilde{q}}F^\R$.
		 Moreover, the transformation $\psi_P$ can be chosen to be equal to $\Psi_P=\Psi_Q\circ\psi$, where $Q\in\MM_n(K[[x]])$ satisfies $Q(0)=I_n$ and $\psi$ is a finite composition of regular polynomial or diagonal monomial gauge transformations (with coefficients in $K$).
	%
	\end{theorem}

\section{Proof of the Complex Turrittin's Theorem}
\label{sec:proofs-complex}

Let $A(x)\in\MM_n(L_K)$ by a system with Poincaré rank $q=q(A)$, written as in (\ref{eq:system-develop}). Let $k=k(A)$ be the radiality index.
Since the radial part $A_0+xA_1+\cdots+x^{k-1}A_{k-1}$ is preserved by any gauge transformation, the coefficient $A_k$ is considered as the {\em first significant} matrix of the system. This must be compared with the usual proofs of Turrittin's theorem, where the radial part is ruled out by an {\em exponential shifting} so that $A_k$ becomes the new leading coefficient (and $q$ drops to $q-k$). In our approach, where we stress the finitely determined nature of the transformations, we do not allow the use of exponential shifting, so that the radial part is carried all along the procedure.

Denote $N(n,q,k)=n(q-k)+k$ as in Remark~\ref{rks}, (a) and consider the statement (i)' to be the same as item (i) in Theorem~\ref{th:c-turrittin-poly} but substituting in the second part the truncation order $nq$ by $N(n,q,k)$.

For the proof of items (i)'-(iii) of Theorem~\ref{th:c-turrittin-poly}, we perform, a priori in arbitrary ordering, several  ramifications, or regular polynomial or diagonal monomial gauge transformations. The desired expression of the composition of those transformations required in the different items of the statement will be a consequence of the following lemma, whose proof is straightforward.

\begin{lemma}\label{lm:commutation-transformations}
Fix any field $\kappa$. Let $r\in\N_{\ge 1}$.  Let $P(x)\in\MM_n(\kappa[x])$ and let $\Psi_P$ be the associated polynomial gauge transformation. Then there exists another polynomial gauge transformation $\Psi_{\tilde{P}}$ satisfying $R_r\circ\Psi_P=\Psi_{\tilde{P}}\circ R_r$. In fact, we can take $\tilde{P}(x):=P(x^r)$. In particular, $\Psi_P$ is regular or diagonal monomial iff $\Psi_{\tilde{P}}$ is so.
\end{lemma}

\strut
 Another important tool is the following result, known with the name of {\em Splitting Lemma}, which is valid for any given base field, algebraically closed or not. It permits to reduce the dimension of the system when the first significant  matrix has two disjoint subsets of eigenvalues. It is usually stated in the formal setting (see for instance \cite{Was,Bal,Bar2}), but it has a finitely determined nature in terms of truncations of the system.

 \begin{lemma}[Splitting Lemma]\label{lm:splitting}
 	Let $K$ be any field. With the same notations as above, if $k$ is the radiality index of the system $A$, assume that $k<q$ and that $A_k$ is conjugated to $A^{11}_k\oplus A^{22}_k$, where the characteristic polynomials $\chi_{A^{11}_k}(\lambda)$ and $\chi_{A^{22}_k}(\lambda)$ are coprime, both of positive degrees, say $n_1$ and $n_2$, respectively. Then there exists a formal regular gauge transformation $\Psi_T$, where $T\in\MM_n(K[[x]])$ satisfies $T(0)=I_n$, such that the transformed system $B=\Psi_T[A]$ writes as $B=B^{11}\oplus B^{22}$, where $B^{ii}\in\MM_{n_i}(L_K)$ is a system of dimension $n_i$ for $i=1,2$. Moreover, $q(B)=q$, $k(B)=k$ and, writing $B(x)=x^{-(q+1)}(\sum_{j\ge 0}x^j B_j)$, we have $A_j=B_j$ for $j=0,1,...,k$ and for any $m>k$, the truncation $J_mB$ only depends on $J_{m-k}T$ and $J_mA$. In other words, if $\wt{A}$ is another system with the same Poincaré rank $q(\wt{A})=q$ and satisfying $J_m\wt{A}=J_m A$ then $J_m\Psi_{T^{(m-k)}}[\wt{A}]=J_mB$, where $T^{(m-k)}=J_{m-k}T$.
 \end{lemma}
 
\subsection{Proof of Theorem~\ref{th:c-turrittin-poly}, (i)'. Getting a $(TRS)$-form of degree $0$ and some rank $\tilde{q}$}\label{sec:c-thm-i}

First, notice that the cases $q=0$ and $q=k$ (the former being a particular case of the later, by definition) are trivial: in these cases, the system is already in $(TRS)_0^q$-form.

We proceed by induction on the dimension $n$ of the system. The starting case $n=1$ is also trivial. Assume then that $n>1$.

\subsubsection{Case with different eigenvalues}\label{sec:case-diff-eigenv}

Suppose that we are in the case where $A_k$ has at least two different eigenvalues. Then we can reduce to a smaller dimension as follows. First, up to a constant regular gauge transformation we can assume that $A_k=A^{11}_k\oplus A^{22}_k$ where $A^{11}_k$ and $A^{22}_k$ are matrices of respective sizes $n_1, n_2$, both smaller than $n$, and having no common eigenvalue. Using Lemma~\ref{lm:splitting} for $N=N(n,q,k)$, there exists a regular polynomial gauge transformation $\Psi_P$ such that the $N$-truncation of $B=\Psi_P(A)$ is written as 
$J_NB=B^{11}\oplus B^{22}$, where, for $i=1,2$, $B^{ii}$ is a system of dimension $n_i$. Moreover, $J_NB$ has the same Poincaré rank, the same radiality index and the same $k$-truncation than $A$. In particular, if $q_i=q(B^{ii})$ and $k_i=k(B^{ii})$ then $q_i\leq q$ and $q_i-k_i\le q-k$. Taking into account that $n_1$ and $n_2$ are both positive and hence strictly smaller than $n$, we obtain for $i=1,2$ that
$$
N(n,q,k)\ge n_i(q-k)+q\ge n_i(q_i-k_i)+q_i\ge n_i(q_i-k_i)+k_i=N(n_i,q_i,k_i).
$$
Using the induction hypothesis to system $B^{ii}$ for $i=1,2$, there is a finite composition $\psi^{ii}$ of transformations in dimension $n_i$ (as in statement (i)) such that $\wt{B}^{ii}=\psi^{ii}(B^{ii})$ is in $(TRS)_0^{q_i}$-form and such that the second part of statement (i)' holds for the truncation order $N(n_i,q_i,k_i)$ in the place of $n_iq_i$. Moreover, in the composition $\psi^{ii}$ there is but a single ramification $R_{r_i}$ (with $r_i\in\N_{\ge 1}$, including the case $R_{r_i}=id$ ($r_i=1$). Now, for $\{i,j\}=\{1,2\}$ write, using Lemma~\ref{lm:commutation-transformations},
$$
R_{r_j}\circ\psi^{ii}=\wt{\vp}^{ii}\circ R_{r_1r_2},
$$
where $\wt{\vp}^{ii}$ is a composition of transformations in dimension $n_i$, either regular polynomial or diagonal monomial (that is, no ramification). Notice that the composition $R_{r_j}\circ\psi^{ii}$ satisfies the requirements of Theorem~\ref{th:c-turrittin-poly}, (i)' for the system $R_{r_j}(B^{ii})$, for which the Poincaré rank and radiality index are equal to $r_jq_i$ and $r_jk_i$, respectively. Writting $\wt{\vp}^{ii}=\Psi_{Q_i}$, where $Q_i\in\MM_{n_i}(K[x])$, we put $r=r_1r_2$ and define
$$
\psi=\Psi_{Q_1\oplus Q_2}\circ R_{r}\circ\Psi_P=\vp\circ R_r,
$$
where $\vp=\Psi_{Q_1\oplus Q_2}\circ\Psi_{P(x^r)}$ is a composition of gauge transformations, either regular polynomial or diagonal monomial. We check that $\psi$ satisfies the requirements of statement (i)' for the initial system $A$. To be convinced, we need to observe two facts. In one hand, for $\{i,j\}=\{1,2\}$, the composition $R_{r_j}\circ\psi^{ii}$ satisfies all the requirements of (i)' for the system $B^{ii}$, since so does $\psi^{ii}$ by construction (notice that if $B$ is any system then we have $q(R_r(B))=rq(B)$, $k(R_r(B))=rk(B)$ and for any $M\ge 0$, the truncation $J_{rM}R_r(B)$ is univocally determined by $J_MB$). On the other hand, use the second part of Lemma~\ref{lm:splitting} to conclude that $J_N\Psi_P(A)$ only depends on $J_NA$ for $N=N(n,q,k)$ and the inequality $N(n_i,q_i,k_i)\le N$ for $i=1,2$ proved above.

\subsubsection{Case with a single eigenvalue}

Suppose now that $A_k$ has a single eigenvalue $\lambda_k\in\overline{K}$. Recall that we are assuming that $K=\overline{K}$. So, up to a constant gauge transformation, we may suppose that $A_k$ is in Jordan normal form. Explicitly, there exists a (unique) sequence $1\le n_1\le n_2\le\cdots\le n_\ell\le n$ with $n=n_1+\cdots+n_\ell$ such that 
\begin{equation}\label{eq:jordan-Ak}
A_k=(\lambda_k I_{n_1}+H^{(n_1)})\oplus\cdots\oplus(\lambda_k I_{n_\ell}+H^{(n_\ell)}),
\end{equation}
where  
$$
H^{(n_j)}=\left(\begin{array}{ccccc}
0 & 1 & & \cdots & 0\\
0 & 0 & 1 & \cdots & 0\\
 &  & \vdots & & \\
0 & 0 & \cdots & 0 & 1\\
0 & 0 & \cdots & 0 & 0	
\end{array}\right)_{n_j\times n_j}
$$
(each such matrix will be called in the sequel a {\em shifting matrix}). Notice also that, since $A_k$ is not a radial matrix, we have $n_j>1$ for at least one index $j$.

 We divide the proof in different steps.

\vspace{.2cm}

{\em Step 1. The tuple $I(A)$.-} In the situation above, for $i=1,...,n$, denote by $\gamma_i(A_k)\in\N_{\ge 0}$ the degree, as a polynomial in $\lambda$, of the g.c.d. of the family of all $i\times i$ minors of the characteristic matrix $A_k-\lambda I_n$. In particular $\g_n(A_k)=n$.
For such a system $A$ (when $A_k$ has a single eigenvalue), we define the following tuple of non-negative integer numbers
$$
I(A):=(\gamma_1(A_k),\cdots,\gamma_n(A_k),q-k).
$$

Remark that each $\gamma(A_k)$ depends only on the conjugation class of $A_k$. We need to recall also the following result on Linear Algebra (see in Wasow \cite[Lemma 19.4]{Was} for a proof) concerning the behaviour of the values $\gamma_i$ for a perturbation of the matrix $A_k$ in the case where $A_k$ has at least two blocks.

\begin{lemma}\label{lm:gamma-i}
Consider $A_k$ as a block-diagonal Jordan matrix $A_k=\left(A_k^{(ij)}\right)$, where the diagonal blocks are given by $A_k^{(ii)}=\lambda_kI_{n_i}+H^{(n_i)}$. Assume that $\ell\ge 2$. With the same block structure, let $G=\left(G^{(ij)}\right)\in\MM_n(K)$ be a block-lower-triangular matrix with the same diagonal blocks as $A_k$ (that is, $G^{(ii)}=A_k^{(ii)}$ and $G^{(ij)}=0$ if $i<j$). Then we have 
$$
\gamma_t(G)\le\gamma_t(A_k)\;\;\;\mbox{ for any }t\in\{1,2,...,n\}
$$ 
and the inequality is strict for at least one $t$ if $G\ne A_k$ (that is, $G^{(ij)}\ne 0$ for at least a pair $(i,j)$ with $i>j$).
\end{lemma}

\vspace{.2cm}

{\em Step 2. Special matrices and choice of $g\in\Q$.-} 
Given a tuple $\sigma=(n_1,n_2,\cdots,n_\ell)$ of positive integers as above, denote $m_j:=\sum_{u=1}^{j}n_u$ for $j=1,...,\ell$. A matrix $T=(t_{uv})\in\MM_n(K)$ will be called a {\em special matrix of type $\sigma$} if $t_{uv}=0$ for every $(u,v)$ such that $u\not\in\{m_1,m_2,...,m_\ell\}$. 

A result in Wasow (\cite[Lemme 19.2]{Was}) assures that for any $N\ge k+1$ there exists a regular polynomial gauge transformation $\Psi_{P_N}$, where $P_N$ is of degree $N$ and $P_N(0)=I_n$, such that the transformed system $B=\Psi_{P_N}(A)$ satisfies that all coefficients of the truncation $J_N(B)$ are special matrices of type $\sigma=(n_1,n_2,\cdots,n_\ell)$, where the $n_j$ are the sizes of Jordan blocks of the matrix $A_k$. We will put $N=N(n,q,k)$ and, renaming $B$ again as the system $A$, we may assume the following assumption
\begin{quote}
	(*)\;\; For any $k+1\le j\le N$, the coefficient $A_j$ is a special matrix of type $\sigma$.
\end{quote} 

Write now
$$
x^{q+1}A(x)=A_0+xA_1+\cdots=
\sum_{i=0}^{k-1}\lambda_ix^iI_n+x^{k}\overline{A},\;
\mbox{ with }\overline{A}=A_k+O(x)=(\overline{a}_{uv}(x)).
$$
Put $\alpha_{uv}:=\nu(\overline{a}_{uv}(x))$. Notice that $\alpha_{uv}\in\N_{\ge 0}$ for any $u,v$, that $\alpha_{uu}>0$ for any $u$ and that $\alpha_{u,u+1}=0$ for at least one index $u\in\{1,...,n\}$. 
\begin{definition}\label{def:g}
With the conditions above, we define the {\em shearing order (of $A$)} as the rational number
$$
g=g(A):=\min\left(\{\frac{\alpha_{uv}}{1+u-v}\,:\,u>v\}\cup\{\alpha_{uu}\,:\,1\le u\le n\}\cup\{q\}\right).
$$
\end{definition}
Geometrically, as discussed in Wasow \cite[pp.104-105]{Was}, the shearing order is the smallest abscissa in which a line of the following family
$$
\{y=(v-u)x+\alpha_{uv}\}_{v<u}\cup\{y=\alpha_{uu}\}_u\cup\{y=q-k\}
$$ cuts the diagonal $y=x$.

Notice also that $g\in\Q$ and that $0<g\le q-k$. Write $g=\frac{h}{r}$ where $h,r$ are positive integers with no common factor.

The name ``shearing'' comes from Wasow's name for the gauge ``ramified'' transformation with matrix $S_g=\text{diag}\,(1,x^g,x^{2g},...,x^{(n-1)g})$. In our definitions, such a transformation has only sense if $g=h$ is an integer. Otherwise, we perform instead the composition $\Psi_{S_h}\circ R_r (=:``\Psi_{S_g}")$.

\vspace{.2cm}

{\em Step 3. The case where $g=h$ is integer}.
We consider the monomial diagonal gauge transformation $\Psi_{S_h}$ where $S_h=\text{diag}\,(1,x^h,x^{2h},...,x^{(n-1)h})$. If $A=(a_{uv}(x))$ then the transformed system $B:=\Psi_{S_h}[A]=S_h^{-1}AS_h-S_h^{-1}S'_h$ writes as
\begin{equation}\label{eq:effect-shearing}
B=\left(\begin{array}{cccc}
	a_{11}(x) & x^h a_{12}(x) & \cdots & x^{(n-1)h}a_{1n}(x) \\	x^{-h}a_{21}(x) & a_{22}(x) & \cdots & x^{(n-2)h}a_{2n}(x) \\
	 & \vdots & & \\
	x^{-(n-1)h}a_{n1}(x) & x^{-(n-2)h} a_{12}(x) & \cdots & a_{nn}(x) 
\end{array}\right)+ x^{-1}K,
\end{equation}
where $K$ is a constant diagonal matrix (in fact $K=\text{diag}(0,h,2h,\dots,(n-1)h)$).

By the choice of the shearing order $g=h$, one can see that we can write $B(x)=x^{-(q+1)}(\wt{B}_0+x\wt{B}_1+\cdots)$, where the $\wt{B}_j$ are constant matrices satisfying
\begin{enumerate}[(a)]
	\item $\wt{B}_j$ is diagonal for $j\in\{0,1,...,k+h-1\}$, with $\wt{B}_j=A_j$ if $j<k$, $\wt{B}_k=\lambda_kI_n$ and $\wt{B}_j=0$ if $k<j<k+h$.
	\item The entries of $\wt{B}_{k+h}$ above the principal diagonal coincide with those of the matrix $H^{(n_1)}\oplus\cdots \oplus H^{(n_\ell)}$. Moreover, $\wt{B}_{k+h}$ has a non-zero entry on or above the principal diagonal, except, possibly, if $h=q-k$. 
	\item For any $M\ge k$, the matrix $\wt{B}_M$ depends only on the truncation $J_{M+(n-1)h}(A)$. 
\end{enumerate}

As a consequence, the Poincaré rank after the shearing transformation does not increase. That is, 
\begin{equation}\label{eq:q-after-shearing}
q(B)\le q.
\end{equation}
Moreover, using (a) and (b), and denoting $q'=q(B)$, $k'=k(B)$, we have that if $q'=q$ then $k'=k+h$, while, if $q'<q$ then $k=0$, $q'=q-h$ and $k'=0$. In any case, we deduce
\begin{equation}\label{eq:qk-prima}
q'-k'<q-k\;\;\mbox{ and }\;\;N(n,q',k')+(n-1)h\le N(n,q,k).
\end{equation}
Thus, using the property (c) and the second part of equation (\ref{eq:qk-prima}), it will suffice to prove item (i)' of Theorem~\ref{th:c-turrittin-poly} for the system $B$. 

Write, accordingly to our main notations, $B(x)=x^{-(q'+1)}\left(B_0+xB_1+\cdots\right)$. We will be done if $B$ is in the trivial case $q'-k'=0$ or if $q'-k'>0$ but $B_{k'}$ has at least two different eigenvalues. Thus, assume that $q'-k'>0$ and that $B_{k'}$ has a single eigenvalue, say equal to $\lambda_{k'}'$. It is enough, by recurrence, to prove in this case that the tuple  $I(B)$ satisfies $I(B)<I(A)$. 

 Notice that $B_{k'}=\wt{B}_{k+h}$, after the discussion above concerning the value of $q'$ and $k'$. Consider the matrix $B_{k'}$ written in blocks according to the Jordan structure of $A_k$: 
$$
B_{k'}=\left(B_{k'}^{(ij)}\right)_{1\le i,j\le\ell},\;\; B_{k'}^{(ij)}\in\MM_{n_j\times n_j}(K).
$$
Using the property (b), we have that $B_{k'}=H^{(n_1)}\oplus\cdots\oplus H^{(n_1)}+T$ where $T$ is a lower triangular matrix. In particular, for $i\in\{1,...,\ell\}$, since $\mbox{Spec}(B_{k'}^{(ii)})=\{\lambda_{k'}'\}$, the matrix $B_{k'}^{(ii)}-\lambda_{k'}' I_{n_i}$ is nilpotent and of rank $n_i-1$. This implies that the Jordan decomposition of $B_{k'}^{(ii)}$ has a single block and hence $B_{k'}^{(ii)}$ is conjugated to $H^{n_i}+\lambda_{k'}'I_{n_i}$. Consequently, $B_{k'}$ is conjugated to a block-lower-triangular matrix with diagonal blocks equal to $H^{n_i}+\lambda_{k'}'I_{n_i}$, those of the matrix $A_k$. Thus, either $B_{k'}$ is conjugated to $A_k$ (when $\ell=1$) or we are in the situation where we can apply Lemma~\ref{lm:gamma-i}. From this lemma, and taking into account the first part of equation (\ref{eq:qk-prima}), we deduce that $I(B)<I(A)$, as wanted. 

\vspace{.2cm}

{\em Step 4. The case where $g$ is not an integer}. Assume now that the shearing order $g$ is not an integer and put $g=\frac{h}{r}$ where $h,r$ are positive integers without common factor and such that $r\ge 2$. We remark that the condition that $g$ is not an integer implies that $g$ is given by one of the quotients $\alpha_{uv}/(1-u-v)$ in Definition~\ref{def:g}, while $\alpha_{uu}>g$ for any $u$, as well as $q-k>g$. 

Consider the ramification $R_r$ and put $\wt{A}=R_r(A)$. This system has Poincaré rank and radiality index equal to $q(\wt{A})=rq$ and $k(\wt{A})=rk$, respectively. Moreover, if we denote $\wt{N}=N(n,rq,rk)$ and $N=N(n,q,r)$, then $\wt{N}=rN$ but the truncation $J_{\wt{N}}(\wt{A})$ only depends on $J_N(A)$, showing that if item (i)' of Theorem~\ref{th:c-turrittin-poly} holds for $\wt{A}$ then it also holds for $A$. 

We notice moreover that the first non-radial coefficient of $\wt{A}$ is equal to $A_k$ and that $\wt{A}$ also satisfies assumption (*) with respect to the same type $\sigma=(n_1,...,n_\ell)$ given by the Jordan structure of $A_k$. On the other hand, the shearing order of $\wt{A}$ is given by $g(\wt{A})=rg=h$, a natural number (the new valuations $\alpha_{uv}$ for $\wt{A}$ are all multiplied by $r$). Hence we are, for $\wt{A}$, in the situation of step 3. However, the last component $q-k$ of the tuple $I(\wt{A})$ has increased and we can not conclude automatically. 
To finish, we put $B=\Psi_{S_h}(\wt{A})$, we denote by $q',r'$ the Poincaré rank and the radiality index of $B$, respectively, and, writing $B=x^{-(q'+1)}\left(B_0+xB_1+\cdots\right)$, we show again that assuming that $q'-k'>0$ and that $B_{k'}$ has a single eigenvalue, we have $I(B)<I(A)$. 

Write $B_{k'}=\left(B^{(ij)}_{k'}\right)$ as a block matrix in the same block structure as $A_k$. We have (cf. property (b) in step 3 above) that $B_{k'}$ is block-lower triangular with diagonal blocks given by $B^{(ii)}_{k'}=H^{(n_i)}+T_i$, where $T_i$ is a lower triangular matrix of size $n_i$. Moreover, we must have that all elements in the diagonal of $T_i$ are zero (since all values $\alpha_{uu}$ for $\wt{A}$ are greater than $h=rg$, as mentioned above). On the other hand, by assumption (*), the entries of $T_i$ on any row except possibly the last one are also zero. On the other hand, since we have assumed that each block $B^{(ii)}_{k'}$, as the entire matrix $B_{k'}$, has a single eigenvalue, being the trace of $B^{(ii)}_{k'}$ equal to zero, such eigenvalue is equal to zero. That is, each $B^{(ii)}_{k'}$ is nilpotent. But this implies that $T_i=0$ and we conclude that
\begin{equation}\label{eq:Bii=Hni}
B^{(ii)}_{k'}=H^{(n_i)},\;\;\mbox{ for any } i\in\{1,2,...,\ell\}.
\end{equation}
Furthermore, by the definition of the shearing order of $\wt{A}$, we must have at least one non-zero entry below the principal diagonal of the matrix $B_{k'}$. Together with the equation (\ref{eq:Bii=Hni}), this implies that $\ell>1$ and that $B_{k'}$ is a matrix $G$ in the situation of Lemma~\ref{lm:gamma-i} with a non-diagonal block $B^{(ij)}_{k'}\ne 0$ for at least one pair $(i,j)$ with $i>j$. We conclude that $\gamma_t(B_{k'})<\gamma_t(A_k)$ for at least one index $t$ and thus $I(B)<I(A)$, as wanted.

\strut

Let us show now the statement in Remark~\ref{rks}, (c) concerning this item (i). In other words, we have to justify that all along the above process for obtaining a $TRS$-form of degree $0$, the Poincaré rank can only increase after a ramification and never after a shearing transformation $\Psi_{S_g}$ with $g\in\N$. 

First, notice that this property goes through the induction arguments discussed in paragraph~\ref{sec:case-diff-eigenv}. Thus, we may assume that we are in the case where $A_k$ has a unique eigenvalue. If the shearing order $g=g(A)$ is an integer, the required property is already established by equation (\ref{eq:q-after-shearing}). On the contrary, if the shearing order is $g=h/r$, with $r>1$ and $h$ not divisible by $r$, the procedure consists in the shearing $\Psi_{S_h}$ after the ramification $R_r$. We conclude using the same equation (\ref{eq:q-after-shearing}) once we observe that the system $\wt{A}:=R_r[A]$ has as first non-radial term the same matrix $A_k$ and satisfies $g(\wt{A})=h\in\N$ (just check that $\wt{A}$ satisfies already the property (*) and that the new values $\alpha_{uv}$ in Definition~\ref{def:g} are the old ones multiplied by $r$).

\subsection{Proof of Theorem~\ref{th:c-turrittin-poly}, (ii). Getting $(TRS)$-form of higher degree $\mu$ when $C$ is non-resonant}\label{sec:c-thm-ii}

A proof of this item (ii) when $q=0$ can be found for instance in Wasow \cite[Thm. 5.1]{Was}, Coddington-Levinson \cite[Thm. 4.1, Ch IV]{Cod-L} or Balser \cite[Thm. 5]{Bal}. The general case $q>0$ is not really different from those references because the exponential part $D(x)$, being diagonal and commuting with $C$ will play no essential role to achieve (ii). For the sake of completeness, let us indicate here the steps of the general proof.

Assume that the system $A$ is already in $(TRS)^q_0$-form with principal part $D(x)+x^qC$ and write it in the form
$
x^{q+1}Y'=\left(A_{0}+xA_{1}+\cdots\right)Y,
$
so that $$x^{q+1}J_qA=A_0+\cdots+x^qA_q=D(x)+x^qC.$$ 

We will need the following result from Linear Algebra (which is actually the core of the proof of the Splitting Lemma~\ref{lm:splitting}). See \cite[Thm. 4.1]{Was} for a proof.

\begin{lemma}\label{lm:linear-algebra}
Fix a field $K$ and let $R,S$ be two square matrices with coefficients in $K$ and of sizes $n\times n$ and $m\times m$, respectively. Assume that $R,S$ have no common eigenvalue in the algebraic closure $\overline{K}$ of $K$. Then the linear map
$X\mapsto RX-XS
$ from $
\MM_{n\times m}(K)$ to itself	
is an isomorphism.
\end{lemma}

Up to reorder the variables of $Y$, we write %
\begin{equation}\label{eq:finer-block}
D(x)=D^{11}(x)\oplus D^{22}(x)\oplus\cdots\oplus D^{\ell\ell}(x),
\end{equation}
 where  $D^{jj}(x)$ is a radial matrix of size $n_j$ (that is $D^{jj}(x)=Q_j(x)I_{n_j}$, where $Q_j(x)\in K[x]_{q-1}$) for each $j=1,..,\ell$, and such that $D^{ii}(x)\ne D^{jj}(x)$ if $i\ne j$. Write also the residual matrix $C$, as well as any coefficient matrix $A_j$ with $j\ge q+1$, as block matrices $C=(C^{uv})_{1\le u,v\le\ell}$ and $A_j=(A_j^{uv})_{1\le u,v\le\ell}$, where $C^{uv},A_j^{uv}\in\MM_{n_u\times n_v}(K)$. Using the commutativity property $[D(x),C]=0$, the assumption $D^{ii}(x)\ne D^{jj}(x)$ for $i\ne j$ and Lemma~\ref{lm:linear-algebra}, we obtain that 
$$
C^{uv}=0\;\;\mbox{ if }u\ne v.
$$ 
We eliminate all coefficients $A_j$ for $j\ge q+1$ by means of a regular formal transformation in two steps.

\strut

{\em First step. We eliminate all non-diagonal blocks $\{A_j^{uv}\}_{u\ne v}$.}
We proceed by induction with respect to $q=q(A)$. If $q=0$ (that is, $D(x)=0$), the block structure is the trivial one and there are no non-diagonal blocks, so that there is nothing to prove. 

Suppose that $q>0$. We consider a coarser block structure %
\begin{equation}\label{eq:coarser-block}
	D(x)=\overline{D}^{11}(x)\oplus\cdots\oplus
\overline{D}^{\ell_1\ell_1}(x)
\end{equation}
in such a way that $\overline{D}^{jj}(0)$ is a radial constant matrix for every $j=1,..,\ell_1$ and $\overline{D}^{jj}(0)\ne\overline{D}^{ii}(0)$ if $i\ne j$. Notice that, up to reordering, each block $\overline{D}^{jj}(x)$ is formed by several of the diagonal blocks of the decomposition (\ref{eq:finer-block}).
We consider a matrix $\overline{T}(x)\in\MM_n(K[[x]])$ of the form $\overline{T}(x)=I_n+x^{q+1}\overline{T}_{q+1}+
x^{q+2}\overline{T}_{q+2}+\cdots$ such that, writing each coefficient $\overline{T}_j=(\overline{T}_j^{uv})$ in the same block structure as the one in (\ref{eq:coarser-block}), we have $\overline{T}_j^{uu}=0$ for any $u\in\{1,...,\ell_1\}$. The  regular gauge transformation $\psi_{\overline{T}}$ transforms the system $A$ into a system $\overline{B}:=\psi_{\overline{T}}[A]$ which is also in $(TRS)^q_0$-form with the same principal part $D(x)+x^qC$. Write $x^{q+1}\overline{B}=D(x)+x^qC+\sum_{j\ge q+1}x^{j}\overline{B}_j$, and, for any $j$, consider each coefficient of $A$ or $\overline{B}$ written in the block structure (\ref{eq:coarser-block}) as $A_j=(\overline{A}_j^{uv})$ and $\overline{B}_j=(\overline{B}_j^{uv})$, respectively. We obtain recursively, for any couple of indices $u,v\in\{1,...,\ell_1\}$ with $u\ne v$:
\begin{equation}\label{eq:recursive-splitting}
	\begin{array}{c}
		\overline{B}_{q+1}^{uv}=
		\overline{A}_{q+1}^{uv}+
		\overline{D}_0^{uu}\overline{T}_{q+1}^{uv}-
		\overline{T}_{q+1}^{uv}\overline{D}_0^{vv} \\
		\overline{B}_{q+2}^{uv}=
		\overline{A}_{q+2}^{uv}+
		\overline{D}_0^{uu}\overline{T}_{q+2}^{uv}-
		\overline{T}_{q+2}^{uv}\overline{D}_0^{vv}+\theta^{uv}_2(A_0,A_1,\overline{T}_{q+1})\\
		\vdots \\
		\overline{B}_{q+j}^{uv}=
		\overline{A}_{q+j}^{uv}+
		\overline{D}_0^{uu}\overline{T}_{q+j}^{uv}-
		\overline{T}_{q+j}^{uv}\overline{D}_0^{vv}+\theta^{uv}_j(\{A_s,\overline{T}_{q+s}\}_{s<j})\\
		\vdots	
	\end{array}
\end{equation} 
where each $\theta^{uv}_j$ is a polynomial matrix acting on the block entries of the explicitly indicated matrices. Using Lemma~\ref{lm:linear-algebra} with $R=\overline{D}^{uu}_0$ and $S=\overline{D}^{vv}_0$, we obtain recursively blocks $\overline{T}^{uv}_{q+1},\overline{T}^{uv}_{q+2},\ldots$ (i.e. the entire matrix $\overline{T}(x)$) so that $\overline{B}^{uv}_j=0$ for every $j\ge q+1$ and for every $u,v\in\{1,...,\ell_1\}$ with $u\ne v$. In other words, $\overline{B}=\psi_{\overline{T}}[A]$ is a block-diagonal formal system $\overline{B}=\overline{B}^{11}\oplus\cdots\oplus \overline{B}^{\ell_1\ell_1}$. Now, we consider separately each one of the subsystems $\overline{B}^{jj}$, for $j=1,...,\ell_1$. By construction, this system is in (TRS)-form of degree $0$ and with exponential part equal to $\overline{D}^{jj}(x)=\overline{D}^{jj}_0+\wt{D}^{jj}(x)$. Put $\wt{B}^{jj}:=\overline{B}^{jj}-x^{-(q+1)}\overline{D}^{jj}_0$, a system with Poincaré rank equal to $q(\wt{B}^{jj})=q-\nu(\wt{D}^{jj}(x))<q$, and also in (TRS)-form of degree $0$ and exponential part equal to $x^{-\nu_{j}}\wt{D}^{jj}(x)$, where $\nu_{j}:=\nu(\wt{D}^{jj}(x))$. Consider a block decomposition
\begin{equation}\label{eq:j-finer-block}
\wt{D}^{jj}(x)=\wt{D}^{jj,1}(x)\oplus\cdots\oplus\wt{D}^{jj,\ell^j}(x)
\end{equation}
analogous to that of $D(x)$ in equation (\ref{eq:finer-block}); that is, each block in (\ref{eq:j-finer-block}) is a radial matrix and two such blocks are different. By recurrence on $q$, there is a formal regular gauge transformation with matrix $T^j(x)=I+\cdots$ (with size equal to the size of the system $\overline{B}^{jj}$) such that $\Psi_{T^j}[\wt{B}^{jj}]$ is block-diagonal in the same block structure as (\ref{eq:j-finer-block}). Taking into account that $\overline{D}^{jj}_0$ is a radial matrix, the same is true for $\Psi_{T^j}[\overline{B}^{jj}]$. Put $$T(x)=(T^1(x)\oplus\cdots\oplus T^{\ell_1})\cdot\overline{T}(x)=I+\cdots\in\MM_{n\times n}(K[[x]]).
$$
Then, $B:=\Psi_{T}[A]$ is a system in $(TRS)^q_0$-form with the same principal part as $A$ and block-diagonal with respect to the concatenation of the different block-structures given by (\ref{eq:j-finer-block}); i.e., $D(x)=\bigoplus_{j=1}^{\ell_1}\wt{D}^{jj}(x)=\bigoplus_{j=1}^{\ell_1}\bigoplus_{i=1}^{\ell^j}\wt{D}^{jj,i}.
$
We are done, since this last decomposition is, up to reordering, the same as the initial one (\ref{eq:finer-block}).

\strut

{\em Step 2. Eliminating the diagonal terms.}
Consider the block-diagonal system $B=\Psi_T[A]=B^{11}\oplus\cdots\oplus B^{\ell\ell}$ as in the step 1, in the block structure given by (\ref{eq:finer-block}). We propose a second formal regular gauge transformation associated to $U\in\MM_n(K[[x]])$ of the form
$$
U=(I_{n_1}+U^{11})\oplus\cdots\oplus(I_{n_\ell}+U^{\ell\ell}),
$$ 
where $U^{uu}$ has size $n_u$ and $U^{uu}(0)=0$ for any $u$. Write $U^{uu}=xU^{uu}_1+x^2U^{uu}_2+\cdots$. The transformed system $E:=\psi_U[B]$ is again block-diagonal and $x^{q+1}J_qE=D(x)+x^qC$. We write $E=E^{11}\oplus\cdots\oplus E^{\ell\ell}$ and $x^{q+1}E^{uu}=D^{uu}+x^qC^{uu}+x^{q+1}E^{uu}_{q+1}+\cdots$ for each block $E^{uu}$. Taking into account that each block $D^{uu}$ is radial and hence commutes with any matrix, we obtain recursively for any $u\in\{1,...,\ell\}$ and for any $j\ge 1$
\begin{equation}\label{eq:recursive-diagonal}
E^{uu}_{q+j}=B^{uu}_{q+j}+C^{uu}U^{uu}_{j}-U^{uu}_j
(C^{uu}+jI_{n_u})+
	\eta^{uu}_j(\{U^{uu}_s,B^{uu}_{q+s}\}_{s<j}),
\end{equation}  
where, similarly as in equation (\ref{eq:recursive-splitting}), $\eta^{uu}_j$ is a polynomial acting over the expressed matrices, with $\eta^{uu}_{1}=0$. Using again Lemma~\ref{lm:linear-algebra} with matrices $R=C^{uu}$ and $S=C^{uu}+jI_{n_u}$ (notice that these two matrices have no common eigenvalue by the non resonance condition for $C$), we can construct recursively from (\ref{eq:recursive-diagonal}) the coefficients $U^{uu}_{q+1},U^{uu}_{q+2},\ldots$ such that 
$
E^{uu}_{q+j}=0
$
 for any $u\in\{1,...,\ell\}$ and for any $j\ge 1$.

 \strut
 
 The conclusion from the two steps above is that the composition $\Psi_{UT}=\Psi_U\circ\Psi_T$ satisfies
 $$
 \Psi_{UT}[A]=x^{-(q+1)}(D(x)+x^qC).
 $$
 To finish, notice from the expression of the transform of a system by a regular gauge transformation that, for every $\mu\ge0$, 
 the $(q+\mu)$-jet of $\Psi_{UT}[A]$ only depends on the $(q+\mu)$-jet of the product $UT$ (and on $J_{q+\mu}A$). 
 Thus, we put $P^\mu:=J_{q+\mu}(UT)$ and Theorem~\ref{th:c-turrittin-poly}, (ii) follows.

\subsection{Proof of Theorem~\ref{th:c-turrittin-poly}, (iii). Getting a non-resonant residual matrix}\label{sec:c-thm-iii}

A proof of this item, for $q=0$, can be found in the same references cited in paragraph~\ref{sec:c-thm-ii} above.

We assume that $A$ is in $(TRS)^q_0$ form with exponential part equal to $D(x)$ and residual part equal to $C$. 

\vspace{.2cm}

{\em Radial case.-} We consider first the case where $D(x)$ is a radial matrix; i.e., $D(x)=Q(x)I_n$, where $Q(x)$ is a polynomial (including the case $q=0$ for which $D(x)=0$). 

Take the partition ${\rm Spec}(C)=\Omega_1\cup\cdots\cup\Omega_r$ of the spectrum of $C$ in such a way that two eigenvalues differ by a non-zero integer number if and only if they belong both to some $\Omega_j$. Denote $\Omega_j=\{\lambda^j_1,...,\lambda^j_{s_j}\}$ and assume that the indices are chosen so that
$
Re(\lambda^j_1)>Re(\lambda^j_2)>\cdots>
Re(\lambda^j_{s_j})
$ in case $s_j>1$.
Let 
\begin{equation}\label{eq:mC}
m(C):=\sum_{\{j\,:\,s_j>1\}}\lambda_1^j-\lambda^j_{s_j}\in\N.
\end{equation}
 Notice that $m(C)>0$ if and only if $C$ is resonant. In order to prove item (iii) in this case, we show that, if $m(C)>0$, there is a constant regular transformation $\Psi_{P_0}$ and a monomial diagonal transformation $\Psi_{S}$ such that the transformed system $\Psi_S\circ\Psi_{P_0}[A]$ has Poincaré rank equal to $q$, the same exponential part $D(x)$ and a residual matrix $C'$ satisfying $m(C')<m(C)$. 

The matrix $P_0$ is chosen so that $\overline{C}:=P_0^{-1}CP_0$ is block diagonal of the form
\begin{equation}\label{eq:block-C}
\overline{C}=(C^{11}\oplus\cdots\oplus C^{1s_1})
\oplus\cdots\oplus(C^{r1}\oplus\cdots\oplus C^{rs_r}),
\end{equation}
where ${\rm Spec}(C^{ji})=\{\lambda^j_i\}$. The transformed system $\overline{A}=\Psi_{P_0}[A]=P_0^{-1}AP_0$ has the same exponential part $D(x)$ (since this is a radial matrix) and residual matrix equal to $\overline{C}$. Put $x^{q+1}\overline{A}=\sum_{l\ge 0}\overline{A}_lx^l$ and use the block structure given by equation (\ref{eq:block-C}) for each coefficient $\overline{A}_l=(\overline{A}_{l}^{uv})$, where $u,v$ run in the set $\{(ji)\,:\,j=1,..,r,i=1,...,s_j\}\cup\{\Lambda\}$. Assume for instance that $\Omega_1$ has at least two elements. The we consider the diagonal monomial matrix
\begin{equation}\label{eq:S}
S:=xI_{t}\oplus I_{n-t},
\end{equation}
where $t$ is equal to the size of $C^{11}$. Consider $B:=\Psi_S[\overline{A}]$ and write $x^{q+1}B=\sum_{l\ge 0}B_lx^l$ and $B_l=(B_l^{uv})$ with the same block structure as in (\ref{eq:block-C}). A calculation shows that
\begin{equation}\label{eq:B_l}
\begin{array}{ll}
B_l^{uu}=\overline{A}_l^{uu},\mbox{ if }l\ne q,&
B_l^{uv}=\overline{A}_l^{uv},\mbox{ if }u\ne (11)\mbox{ and }v\ne(11), \\
B_l^{(11)v}=x^{-1}\overline{A}_l^{(11)v}
,\mbox{ if }v\ne(11),& 
B_l^{u(11)}=x\overline{A}_l^{u(11)}
,\mbox{ if }u\ne(1,1),\\
B_q^{(11)(11)}=\overline{A}_q^{(11)(11)}-I_t. &
\end{array}
\end{equation}
In particular, we obtain $q(B)=q$, $B$ is in $(TRS)^q_0$-form with the same exponential part $D(x)$ and the residual matrix $C'=B_q$ is upper triangular with respect to the block structure $(B_q^{uv})$ and with diagonal equal to 
$$
{\rm diag}(C')=((C^{11}-I_t)\oplus\cdots\oplus C^{1s_1})
\oplus\cdots\oplus(C^{r1}\oplus\cdots\oplus C^{rs_r})\oplus C^\Lambda.
$$ 
We deduce that $m(C')=m(C)-1$ and we are done.

\vspace{.2cm}

{\em General case.-} Notice that, in preceding case, the degree of a polynomial gauge transformation needed to obtain a non-resonant residual matrix can be bounded by $m(C)$ (cf. Remark~\ref{rks}, (e)). Moreover, such transformation depends only on the truncation $J_{m(C)+q}(A)$. We use this remark and Step 1 in paragraph~\ref{sec:c-thm-ii} to reduce the general case to the precedent case. More precisely, consider the decomposition of the exponential part in radial matrices $D(x)=D^{11}(x)\oplus\cdots\oplus D^{\ell\ell}(x)$, as in equation (\ref{eq:finer-block}). Consider also $C=C^{11}\oplus\cdots\oplus C^{\ell\ell}$, decomposed into the same block structure. Put
$$
m:=\max\{m(C^{ii}),\:\,i=1,...,\ell\}.
$$
Taking into account Step 1 in the proof of (ii), there is regular polynomial gauge transformation $\psi$ such that the truncation $J_{m+q}(\psi[A])$ decomposes into several systems of smaller size
$$
J_{m+q}(\psi[A])=A^{11}(x)\oplus\cdots\oplus A^{\ell\ell}(x),
$$
where the principal part of $A^{jj}(x)$ is $x^{-\nu(D^{jj})}(D^{jj}(x)+x^qC^{jj})$. In particular, $A^{jj}(x)$ is in the radial case treated above, so that there is a finite composition of constant regular and monomial diagonal transformations $\vp_j$ such that $\vp_j[A^{jj}]$ has the same exponential part as $A^{jj}(x)$ and a non-resonant residual matrix. We conclude that the composition $\phi=(\vp_1\oplus\cdots\oplus\vp_\ell)\circ\psi$ transforms $A$ into a system in $(TRS)^q_0$-form with non-resonant residual matrix. On the other hand, as we have noticed in the radial case, the degree of each $\vp_j$ is bounded by $m(C^{jj})$, and hence the degree of the polynomial gauge transformation $\phi$ can be bounded by $2m$. This ends the proof of item (iii) of Theorem~\ref{th:c-turrittin-poly} and completes the statement in Remark~\ref{rks}, (e).  \hfill{$\square$}

\strut

One final comment on how to conclude Remark~\ref{rks}, (d) in what it concerns for this item (iii). We need to take into account that the diagonal monomial transformations used in the process are only those of the form $\Psi_S$, where $S$ is as in equation (\ref{eq:S}). As we have already observed from equations (\ref{eq:B_l}), such transformations preserve the Poincaré rank. 

\section{Proof of the Real Turrittin's Theorem}

In this section we fix a real closed field $K$ and we prove Theorem~\ref{th:r-turrittin-poly}, the real version of Turrittin's Theorem on polynomial normal forms. As mentioned, the formal statement Theorem~\ref{th:r-turrittin-formal} will be a consequence of it.

We make use of the monomorphism of $K$-algebras defined in paragraph~\ref{sec:statement-real}. That is, for any $m\in\N$, we consider 
$$
\Theta_m:\MM_m(\overline{K})\to\MM_{2m}(K),\;\;(c_{uv})\mapsto(\Lambda_{c_{uv}}).
$$
and (with the same name), its extension to a morphism of $K$-algebras from $\MM_m(L_{\overline{K}})$ to $\MM_{2m}(L_K)$ sending an $m$-dimensional system $B=x^{-(q+1)}\sum_{j\ge 0}x^j B_j$ with coefficients in $\overline{K}$ to the system $\Theta_m(B)=x^{-(q+1)}
\sum_{j\ge0}x^j\Theta_m(B_j)$. 
Notice that $\Theta_m$ preserves the Poincaré rank but not necessarily the radiality index of the system. On the other hand, one can check easily that $\Theta_m$ commutes with the gauge transformations and with ramifications. To be precise, if $B,P\in\MM_m(L_{\overline{K}})$ with $\text{det}(P)\ne0$, we have
\begin{equation}\label{eq:gauge-phim}
\Theta_m(\Psi_P[B])=\Psi_{\Theta_m(P)}[\Theta_m(B)],	
\end{equation}
and, if $r$ is a natural non-zero number, then
\begin{equation}\label{eq:ramification-phim}
	\Theta_m\circ R_r=R_r\circ\Theta_m.	
\end{equation}
\subsection{Propagating a $\C$-matrix to higher order coefficients}

The key result for the proof of Theorem~\ref{th:r-turrittin-poly} is the following one.

\begin{proposition}\label{pro:key-for-real}
	Consider a system $A\in\MM_n(L_K)$ with Poincaré rank equal to $q$ and written as $A=x^{-(q+1)}\left(A_0+xA_1+\cdots\right)$. Let $k$ be the radiality index of $A$ and assume that $k<q$ and that the spectrum of $A_k$ in $\overline{K}$ consists in a pair of conjugated values $a\pm ib$ with $a,b\in K$ and $b\ne0$ (thus in particular $n=2m$ is even). Then there exists a formal regular gauge transformation $\Psi_T$, where $T\in\MM_n(K[[x]])$, such that the transformed system $B=\Psi_T[A]$ is a $\C$-system. Moreover, writing $B=x^{-(q+1)}(\sum_{j\ge 0}x^j B_j)$, we have $A_j=B_j\in KI_n$ for $j=0,1,...,k-1$ and for any $\mu\ge k$, the truncation $J_\mu B$ only depends on $J_{\mu-k}T$ and $J_\mu A$. In other words, if $\wt{A}$ is another system with $q(\wt{A})=q$ and satisfying $J_\mu\wt{A}=J_\mu A$, then $J_\mu(\Psi_{J_{\mu-k}T}[\wt{A}])=J_\mu B$.
\end{proposition}

The proof of Proposition~\ref{pro:key-for-real} has a big similarity with the one of the Splitting Lemma (cf. Lemma~\ref{lm:splitting}) or of Theorem~\ref{th:c-turrittin-poly}, (ii). This time, it is based on the following result for $\C$-matrices of size two:

\begin{lemma}\label{lm:C-matrices}
Let $\lambda\in\overline{K}\setminus K$ and let $\Lambda_\lambda=\Theta_1(\lambda)$ be the corresponding $\C$-matrix of size $2$. Given $S\in\MM_2(K)$ an arbitrary matrix with coefficients in $K$, there exists a matrix $X\in\MM_2(K)$ such that 
$
\Lambda_\lambda X-X\Lambda_\lambda+S
$
is a $\C$-matrix.
\end{lemma}
\begin{proof}
	Put $\lambda=a+ib$ with $a,b\in K$ and $b\ne0$ and write $S=(s_{ij})$ and $X=(x_{ij})$ with $1\le i,j\le 2$. Computing we have
	\begin{equation}\label{eq:lambdaX-Xlambda}
	\Lambda_\lambda X-X\Lambda_\lambda+S=\left(\begin{array}{cc}
		-u+s_{11} &  v+s_{12} \\
		v+s_{21} &  u+s_{22} 
	\end{array}\right),
	\end{equation}
	where $u=b(x_{12}+x_{21})$ and $v=b(x_{11}+x_{22})$. The matrix in (\ref{eq:lambdaX-Xlambda}) is a $\C$-matrix iff we have $-u+s_{11}=u+s_{22}$ and $v+s_{12}=-(v+s_{21})$. These two last equations have solutions in $u,v$ once we are given the entries $s_{ij}$ of $S$ and, taking into account that $b\ne 0$, we conclude the lemma. 
\end{proof}

{\em Proof of Proposition~\ref{pro:key-for-real}.-} First, using a real canonical form of $A_k$, there is a non-singular matrix $T_0$ with entries in $K$ such that $T_0^{-1}A_kT_0=\Lambda+H$ where
$$
\Lambda=\Lambda_\lambda\oplus\Lambda_\lambda\oplus\cdots\oplus\Lambda_\lambda,\;
H=\left(\begin{array}{ccccc}
	0 & \epsilon_1 I_2 & & & \\
	  & 0  &\epsilon_2 I_2  & & \\
	 & & \ddots & & \\ 
	 & & & & \epsilon_{m-1} I_2 \\
	 & & & & 0
\end{array}\right), \mbox{ with }\epsilon_j\in\{0,1\}.
$$
Up to replacing $A$ by $\Psi_{T_0}[A]$, we may assume that the original coefficient $A_k$ has already the form above $A_k=\Lambda+H$, a $\C$-matrix. We look for a regular formal gauge transformation $\Psi_{T}$ with $T=I_n+xT_1+x^2T_2+\cdots$ satisfying the required property. For that, we compute the coefficients of the transformed system $B:=\Psi_T[A]$ in terms of the coefficients of $A$ and of $T$. With similar computations as already done in the preceding section, if we write $B=x^{-(q+1)}\left(B_0+xB_1+\cdots\right)$ then we get:

- The radial part does not change; i.e., $B_\ell=A_\ell$ for $\ell=0,1,...,k-1$.

- $B_k=A_k=\Lambda+H$.

- For $j\ge k+1$, we obtain
\begin{equation}\label{eq:Bj-in-real}
	B_j=[A_k,T_{j-k}]+A_j+Q_j,
\end{equation}
where $Q_j$ is a matrix which depends polynomially only on the matrices of the family  $\{A_{k+s}, T_{s}\}_{s<j-k}$. 

Let us show that we can choose recursively $T_1,T_2,\ldots$ such that each $B_j$ in equation (\ref{eq:Bj-in-real}) is a $\C$-matrix for any $j\ge k$. This will finish the proof of Proposition~\ref{pro:key-for-real}.

The starting case $j=k$ is done since $B_k=A_k$ is already a $\C$-matrix. Suppose that for $j>k$ we have already constructed $T_1,...,T_{j-1}$ such that $B_\ell$ is a $\C$-matrix for $\ell<j$. For each value of the letter $Y\in\{A,T,B,Q\}$ and for each $\ell\le j$, we write $Y_\ell=(Y_\ell^{uv})_{1\le u,v\le m}$ in a block structure of $2\times 2$ matrices. We construct the different blocks $T_j^{uv}$ in the following order.
We start by the bottom of the first column: the block $T^{m1}_j$ satisfies, after equation (\ref{eq:Bj-in-real}),
$$
[\Lambda_\lambda,T_j^{m1}]+A_j^{m1}+Q_j^{m1}=B_j^{m1}.
$$
Using Lemma~\ref{lm:C-matrices}, we choose $T_j^{m1}$ in such a way that $B_j^{m1}$ is a $\C$-matrix. Then we continue with the block $T_j^{m-1,1}$ which satisfies
$$
[\Lambda_\lambda,T_j^{m-1,1}]+\epsilon_{m-1}T_j^{m1}+A_j^{m-1,1}+
Q_j^{m-1,1}=B_j^{m-1,1}.
$$
Taking into account that $T_j^{m1}$ has already been chosen and using Lemma~\ref{lm:C-matrices}, we choose $T^{m-1,1}_j$ such that $B^{m-1,1}_j$ is a $\C$-matrix. The process can be repeated in this way until we construct all blocks in the first column, that is, those of the form $T_j^{u1}$ in inverse order for $u$ from $u=m$ to $u=1$. After that, we construct the blocks in the second column $T_j^{u2}$, again from $u=m$ to $u=1$: by (\ref{eq:Bj-in-real}) we get
$$
[\Lambda_\lambda,T_j^{u,2}]+
\epsilon_{u}T_j^{u+1,2}-\epsilon_1T_j^{u,1}+A_j^{u2}+
Q_j^{u2}=B_j^{u2}
$$
(with $\epsilon_m=0$), and we choose $T_j^{u2}$ such that $B^{u2}_j$ is a $\C$-matrix, once the blocks $T^{u+1,2}_j$, $T^{u1}_j$ in the above equation are already known. We continue in this way in order to complete the construction of all blocks $T^{uv}_j$ so that any block $B^{uv}_j$ (and hence the whole matrix $B_j$) is a $\C$-matrix.  \hfill{$\square$} 

\subsection{Proof of Theorem~\ref{th:r-turrittin-poly}, (i). Getting a $(\R TRS)$-form of degree $0$}

Fix a singular system $A\in\MM_n(L_K)$ with Poincaré rank $q=q(A)$ and write $A=x^{-(q+1)}\left(A_0+xA_1+\cdots\right)$ as in (\ref{eq:system-develop}) with $A_0\ne 0$. Denote by $k=k(A)$ the radiality index of $A$. As in the complex case, we prove a slightly improvement of item (i) in Theorem~\ref{th:r-turrittin-poly} (called item (i') in what follows), where the sufficient jet order $nq$ to obtain the same real Turrittin-Ramis-Sibuya form is replaced by the order $N=N(n,q,k):=n(q-k)+k$.

\vspace{.1cm}

We start with the trivial case where $q=k$ (this includes the case $q=0$). Using the real Jordan canonical form of $A_k=A_q$, there is a non-singular matrix $T_0\in\MM_n(K)$ such that $T_0^{-1}A_0T_0=C_1\oplus C_2$, where $C_1$ is a matrix with eigenvalues in $K$ and $C_2$ is a $\C$-matrix. The radial part $A_0+xA_1+\cdots+x^{q-1}A_{q-1}$ is preserved by $\Psi_{T_0}$ and can be written in the form $D_1(x)\oplus\Theta_{n_2}(D_2(x))$ where $D_1(x)\in\MM_{n_1}(K[X])$ and $D_2(x)\in\MM_{n_2}(K[x])$ are both diagonal polynomial. Hence $\Psi_{T_0}[A]$ is in $(\R TRS)^{q}_0$-form and (i') follows (notice that $N=q$ in this case). 
 
\vspace{.2cm}

Assume that $0\le k<q$. We proceed by induction with respect to the size $n$ of the system. The case $n=1$ is also trivial: $A$ is already in $(\R TRS)^q_0$-form with $n_1=n=1$ and $n_2=0$ and $N=q$ in this case. Suppose then that $n>1$.

\vspace{.1cm}

Suppose first that the first non-radial term $A_k$ has at least two non-conjugated eigenvalues in $\overline{K}$. In this case, after a constant regular gauge transformation $\Psi_{T_0}$, with $T_0\in GL_n(K)$, we can assume that $A_k=A^{11}_k\oplus A^{22}_k$, where each $A^{ii}_k$ is a square matrix with positive size $n_i$ with entries in $K$ and $\text{Spec}(A^{11}_k)\cap\text{Spec}(A^{22}_k)=\emptyset$. Apply the Splitting Lemma to $A$ up to order $N=N(n,q,k)$. That is, there exists a regular polynomial gauge transformation $\Psi_T$ where  $T=I+xT_1+\cdots+x^{N-k}T_{N-k}$ such that $J_N(\Psi_T[A])=B^{11}\oplus B^{22}$, where $B^{ii}$ is a (polynomial) system of size $n_i<n$ with coefficients in $K$. By induction on the size, item (i') holds for both systems $B^{ii}_k$. In a way completely analogous as we did for the complex case in paragraph~\ref{sec:c-thm-i}, we use this to conclude item (i') for the original system $A$.

\vspace{.1cm}

Suppose now that $\text{Spec}(A_k)=\{\lambda_k,\overline{\lambda_k}\}$ for some $\lambda_k\in\overline{K}$. We consider the two possible situations:

\strut

{\em Case 1: $\lambda_k\ne\overline{\lambda_k}$.} Notice that $n=2m$ is even in this case. We apply Proposition~\ref{pro:key-for-real} to the system $A$. Notably, let $\Psi_T$ be a formal regular gauge transformation with $T\in\MM_n(K[[x]])$ such that $B=\Psi_T[A]$ is a $\C$-system. Let $\overline{B}$ be the system with coefficients in $\overline{K}$ satisfying $B=\Theta_m(\overline{B})$. Apply Theorem~\ref{th:c-turrittin-poly}, (i) to $\overline{B}$: we get a natural number $r\ge 1$ and gauge transformations $\vp_1,\ldots,\vp_s$, either regular polynomial or monomial diagonal (with coefficients in $\overline{K}$) such that, putting $\psi=\vp_s\circ\ldots\circ\vp_1\circ R_r$, we have
\begin{enumerate}[(a)]
\item The system $\psi[\overline{B}]$ is in $(TRS)^{\tilde{q}}_0$-form for some $\tilde{q}\ge 0$. 
\item Being $N=N(n,q,k)$, if $\overline{E}$ is another system with $q(\overline{E})=q$ and $J_N\overline{E}=J_N\overline{B}$, the system $\psi[\overline{E}]$ is also in $(TRS)^{\tilde{q}}_0$-form with the same principal part as $\psi[\overline{B}]$.
\end{enumerate}

Now, for any $i=1,...,s$, if $\vp_i=\Psi_{P_i}$ with $P_i\in\MM_n(\overline{K}[[x]])$, we put $\wt{\vp}_i:=\Psi_{\Theta_m(P_i)}$. Notice that $\wt{\vp}_i$ is a gauge transformation with coefficients in $K$, either regular polynomial or diagonal monomial. Denote by
$\wt{\psi}=\wt{\vp}_s\circ\cdots\circ\wt{\vp}_1
\circ R_r$ and let us see that the composition 
$$
\wt{\psi}\circ\Psi_{J_NT}=\wt{\vp}_s\circ\cdots\circ\wt{\vp}_1
\circ\Psi_{J_NT(x^r)}\circ R_r
$$ satisfies the requirements of Theorem~\ref{th:r-turrittin-poly}, (i'). Using the property (b) above and equations (\ref{eq:gauge-phim}) and (\ref{eq:ramification-phim}) we have that 
$$
\Theta_m(\psi[J_N(\overline{B})])=
\wt{\psi}[J_N(\Theta_m(\overline{B}))]=\wt{\psi}[J_N(B)]
$$   
is in $(\R TRS)^{\tilde{q}}_0$-form. On the other hand, we have the following

\vspace{.2cm}

{\bf Claim.-} If $E\in\MM_n(L_K)$ is any system with Poincaré rank equal to $q$ and $J_NE=J_NB$ then $J_{\tilde{q}}(\wt{\psi}[E])=J_{\tilde{q}}(\wt{\psi}[J_NB])$.

\vspace{.2cm}
Applying this claim to $E=\Psi_{J_NT}[A]$, taking into account that $J_N(\Psi_{J_NT}[A])=J_NB$, we conclude  Theorem~\ref{th:r-turrittin-poly}, (i') in this Case 1.

It remains to show the Claim. It is a consequence of the property (b) above (using (\ref{eq:gauge-phim}) and (\ref{eq:ramification-phim})) in the case that $E$ is a $\C$-system). To be convinced that it is true for any system $E$ in the hypothesis of the statement, we notice, using the description of a general gauge transformation or a ramification, that there exists some integer $M>0$ such that $J_{\tilde{q}}(\wt{\psi}[E])=J_{\tilde{q}}(\wt{\psi}[J_ME])$ for such systems and that the map $H_M:J_ME\mapsto J_{\tilde{q}}(\wt{\psi}[J_ME])$ is a polynomial map in the entries of the coefficient matrices $E_0,...,E_M$ of $J_ME$. Necessarily the minimum $M$ with this property must be greater or equal than $N$. But, as we have just said, property (b) implies that the value of $H_M$ does not depend on the entries of $E_{N+1},\ldots,E_M$ if $E$ is a $\C$-system. Since the set of $J_M$-jets of $\C$-systems with a fixed Poincaré rank $q$ has non-empty interior in the space of $J_M$-jets of all systems (with that fixed Poincaré rank), we conclude that $M=N$ satisfies the property above. The Claim follows.

\vspace{.5cm}

{\em Case 2: $\lambda_k=\overline{\lambda_k}$.} 
In this case, $A_k$ has a single eigenvalue $\lambda_k\in K$. After a constant gauge transformation with entries in $K$, we can write $A_k$ in its Jordan canonical form as in equation (\ref{eq:jordan-Ak}). We can define in this case the  tuple $I(A)=(\g_1(A_k),\ldots,\g_n(A_k),q-k)$ and proceed exactly as in the proof of the complex Turrittin theorem in paragraph~\ref{sec:c-thm-i} from the step in which $A_k$ has the Jordan form (\ref{eq:jordan-Ak}). Notably, the terms in the truncation $J_{N}A$, with $N=N(n,q,k)$, determine a {\em shearing order} $g=h/r\in\Q$ (cf. Definition~\ref{def:g}). Then, we  consider the transformed system $B=\Psi_{S_h}\circ R_r[A]$, where   $R_r$ is the ramification of index $r$ and  $S_h=\text{diag}(1,x^h,...,x^{(n-1)h})$. Denoting by $q'$ and $k'$ the Poincaré rank and the radiality index of $B$ respectively, one of the following situations occurs:
\begin{itemize}
	\item $k'=q$ (including the case $q'=0$): we finish since this the trivial case.
	\item $0\le k'<q'$ and $B_{k'}$ has at least two non-conjugated eigenvalues in $\overline{K}$: we finish using splitting lemma and induction on $n$ as above.
	\item $0\le k'<q'$ and $B_{k'}$ has a unique pair of conjugated eigenvalues that do not belong to $K$: we finish since we are in Case 1 above.
	\item $0\le k'<q'$ and $B_{k'}$ has a unique eigenvalue that belongs to $K$: in this case, the arguments in Steps 3 and 4 in paragraph~\ref{sec:c-thm-i} are valid for the real closed field $K$ and they permit to conclude $I(B)<I(A)$ (in lexicographical order). We finish again since this tuple of non-negative integers number cannot decrease indefinitely. 
\end{itemize}

This ends the proof of Theorem~\ref{th:r-turrittin-poly}, (i'). \hfill{$\square$}

\subsection{Proof of Theorem~\ref{th:r-turrittin-poly}, (ii): getting $(\R TRS)$-form of higher degree}\label{sec:R-item-(ii)}

The proof can be done similarly to the case where $K$ is algebraically closed in paragraph~\ref{sec:c-thm-ii}, with only some minor changes. Let us indicated them.

Suppose that the system $A$ is in $(\R TRS)^q_0$-form with exponential part equal to $D(x)=D_1(x)\oplus D_2(x)$ and residual matrix $C=C_1\oplus C_2$ in the conditions of Definition~\ref{def:TRS-form-real}. In particular, $D_2(x)=\Theta_{n_2}(E_2(x))$ and $C_2=\Theta_{n_2}(G_2)$, where $G_2\in\MM_{n_2}(\overline{K})$ and $E_{2}(x)={\rm diag}(d_1(x),...,d_{n_2}(x))$ with $d_j(x)\in\overline{K}[x]_{q-1}\setminus K[x]$. We also assume that $C$ is non-resonant, which is equivalent to say that both $C_1$ and $G_2$ are non-resonant matrices (of sizes $n_1$ and $n_2$, respectively). 

We consider a block structure to write our system, similar to the one given in (\ref{eq:finer-block}), but compatible with the fact that $D_2(x)$ is a $\C$-matrix. Notably, we write
\begin{equation}\label{eq:block-R}
D(x)=D^{11}(x)\oplus\cdots\oplus D^{ss}(x)\oplus\Phi_{k_1}(E^{11}(x))\oplus\cdots\oplus\Phi_{k_t}(E^{tt}(x)),	
\end{equation}
where 
\begin{itemize}
	\item Each $D^{jj}(x)=f_j(x)I_{m_j}$ is a radial matrix in $\MM_{m_j}(K[x])$ (i.e., the coefficients of $f_j(x)$ are in $K$).
	\item Each $E^{jj}(x)=g_j(x)I_{k_j}$ is a radial matrix in $\MM_{k_j}(\overline{K}[x])$ (i.e., the coefficients of $g_j(x)$ are in $\overline{K}$).
	\item $D_1(x)=D^{11}(x)\oplus\cdots\oplus D^{ss}(x)$ and $D_2(x)=\Phi_{n_2}(E^{11}(x)\oplus\cdots\oplus E^{tt}(x))$.
\end{itemize}  
Put 
In accordance with the notation of equation (\ref{eq:finer-block}), we denote $\ell:=s+t$ and $D^{jj}(x):=\Phi_{k_{j-s}}(E^{j-s,j-s}(x))$ for $j=s+1,...,\ell$. We write also $A=x^{-(q+1)}\sum_l x^lA_l$ as in (\ref{eq:system-develop}) and each coefficient $A_l=(A^{uv}_l)_{1\le u,v\le\ell}$ in the block structure of (\ref{eq:block-R}). 

Notice that the residual matrix $C=C_1\oplus C_2$ is block-diagonal in this structure, since it commutes with $D(x)=D_1(x)\oplus D_2(x)$ and $C_2$ is a $\C$-matrix. Thus, $A_{q+1}^{uv}=C^{uv}=0$ if $u\ne v$. 

We want to eliminate all coefficients $A_{q+1},A_{q+2},...$ by means of a formal regular gauge transformation $\Psi_{P}$ with $P\in\MM_n(K[[x]])$ and $P(0)=I_n$. We proceed as in paragraph~\ref{sec:c-thm-ii} in two steps: first we eliminate the non-diagonal blocks $A_l^{uv}$, for $u\ne v$ and $l\ge q+1$, and then the diagonal blocks $A_l^{uu}$, for $u=1,...,\ell$ and $l\ge q+1$. 

\strut

The first step is proved by induction on $q$, as in the mentioned paragraph. The case $q=0$ is trivial. If $q>0$, we consider a coarser block structure than the one given by (\ref{eq:finer-block}) More precisely, we consider a similar block structure as the one in (\ref{eq:coarser-block}), where each block $\overline{D}^{jj}(x)$ with $j\in\{1,...,\ell_1\}$ is of maximal size such that its value $\overline{D}^{jj}_0:=\overline{D}^{jj}(0)$ at zero is:
\begin{enumerate}[(a)]
	\item Either a radial matrix (i.e., $\overline{D}^{jj}_0=a_j I_{h_j}$ with some $a_j\in K$).
	\item  Or a radial $\C$-matrix (i.e., $\overline{D}^{jj}_0=\Phi_{h_j}((a_j+ib_j)I_{h_j})$ for some $a_j+ib_j\in\overline{K}\setminus K$).
\end{enumerate}
 Notice that a block $\overline{D}^{jj}(x)$ in the case (a) may contain several of the blocks in the decomposition (\ref{eq:block-R}), even of the two different types $\{D^{ll}(x)\}_{l\le s}$ and $\{D^{ll}(x)\}_{l>s}$. In any case, using the same equations (\ref{eq:recursive-splitting}) and Lemma~\ref{lm:linear-algebra}, we can construct a formal matrix $\overline{T}=I_n+x^{q+1}\overline{T}_{q+1}+\cdots\in K[[x]]$ such that the system $\overline{B}=\Psi_{\overline{T}}(A)$ has zero non-diagonal blocks with respect to this last structure  $D(x)=\bigoplus_{1\le j\le\ell_1}\overline{D}^{jj}(x)$. Hence, $\overline{B}=\bigoplus_{1\le j\le\ell_1}\overline{B}^{jj}$, where $\overline{B}^{jj}$ is of size $h_j$ when $\overline{D}^{jj}_0$ is in the case (a), or $\overline{B}^{jj}$ is of size $2h_j$ when $\overline{D}^{jj}_0$ is in the case (b). In this last case, using Proposition~\ref{pro:key-for-real}, we can assume that $\overline{B}^{jj}$ is a $\C$-system (notice that $b_j\ne 0$ in this case, so that $k(\overline{B}^{jj})=0$ and $q(\overline{B}^{jj})=q>0$). 
 
 Put $\wt{B}^{jj}:=\overline{B}^{jj}-x^{-(q+1)}\overline{D}^{jj}_0$ for $j=1,...,\ell_1$, a system with Poincaré rank strictly smaller than $q$. 
 At this point, the proof continues as the one in Step 1 of paragraph~\ref{sec:c-thm-ii} by constructing, using the induction hypothesis, a regular transformation $\Psi_{T^{j}}$ that applies and transform the subsystem $\wt{B}^{jj}$ into a block-diagonal one with respect to the structure induced on the block $\overline{B}^{jj}$ by (\ref{eq:block-R}). We only have to take care about the following: for any index $j\in\{1,...,\ell_1\}$ such that $\overline{D}^{jj}_0$ in the case (b) above, the matrix $T^j$ must be chosen to be a $\C$-matrix, so that $\Psi_{T^{j}}$ preserves the system $x^{-(q+1)}\overline{D}^{jj}_0$ and thus this transformation applied to $\overline{B}^{jj}$ produces the same result.
 
 \strut
 
 Finally, the second step (eliminating the diagonal blocks $A^{uu}_l$ for $l\ge q+1$) is obtained exactly in the same way as in Step 2 of the proof of Theorem~\ref{th:c-turrittin-poly}, (ii) in paragraph~\ref{sec:c-thm-ii}: we have to solve recursively the same equations (\ref{eq:recursive-diagonal}) for the blocks $U^{jj}$, and this can be done independently of the base field $K$, since we only need  Lemma~\ref{lm:linear-algebra} (valid for any field) and the hypothesis that $C$ is non-resonant.

\subsection{Proof of Theorem~\ref{th:r-turrittin-poly}, (iii). Getting a non-resonant matrix}
The proof of this item is made entirely equal to the corresponding complex case (cf. Theorem~\ref{th:c-turrittin-poly}, (iii)) in paragraph~\ref{sec:c-thm-iii}. The only difference is that the first case treated there, called the ``radial case'', must be treated here in two different cases: either we are in the similar ``radial case'' with coefficients in $K$ (that is $D(x)=f(x)I_n$ with some polynomial $f(x)\in K[x]_{q-1}$), or we are in the {\em $\C$-radial case} (that is $D(x)=\Theta_{n/2}(g(x)I_{n/2})$, where $g(x)\in\overline{K}[x]_{q-1}\setminus K[x]$). In the second of these two cases, we have that $D(x)\ne 0$ and $0\le k(A)<q(A)$, so that we can apply Proposition~\ref{pro:key-for-real} and assume, after a regular polynomial gauge transformation of degree $m=m(C)$ (cf. equation (\ref{eq:mC})), that the truncation $J_{q+m}(A)$ is a $\C$-system, image by $\Theta_{n/2}$ of some system $\wt{A}\in\MM_{n/2}(L_{\overline{K}})$ with exponential part equal to $\wt{D}(x)=g(x)I_{n/2}$. By Theorem~\ref{th:c-turrittin-poly}, (iii), we transform $\wt{A}$ into another one with the same exponential part and non-resonant residual matrix by a regular gauge transformation $\Psi_{S}$, where $S=I_{n/2}+xS_1+\cdots\in\MM_{n/2}(\overline{K}[x]_{m})$. In this case, the regular transformation $\Psi_{\Theta_{n/2}(S)}$ proves item (iii) for the real system $A$.

The general case is done as in paragraph~\ref{sec:c-thm-iii} by means of the decomposition (\ref{eq:block-R}) of $D(x)$ and using step 1 of the proof of Theorem~\ref{th:r-turrittin-poly}, (ii), already discussed in the previous paragraph~\ref{sec:R-item-(ii)}.

\strut

This ends the proof of Theorem~\ref{th:r-turrittin-poly}. \hfill{$\square$}

\end{document}